\documentclass{amsart}

\usepackage{xcolor}
\usepackage{setspace}

\newtheorem{theorem}{Theorem}[section]

\newtheorem{lem}[theorem]{Lemma}

\numberwithin{equation}{section}
\title[Lipschitz perturbations]{Eliminating positive-measure level sets by small Lipschitz perturbations}

\author{Sorina Barza}
\address{Department of Mathematics and Computer Science, Karlstad University, Universitetsgatan 2, 65188 Karlstad, Sweden}
\email{sorina.barza@kau.se}

\author{Martin Lind}
\address{Department of Mathematics and Computer Science, Karlstad University, Universitetsgatan 2, 65188 Karlstad, Sweden}
\email{martin.lind@kau.se}

\subjclass[2020]{26A16, 46E35}

\keywords{Level sets, Lipschitz continuity, Lipschitz perturbation}

\begin{document}

\begin{abstract}
    We establish a new regularity phenomenon of continuous functions. Specifically, given any continuous function $f$ and arbitrary $\epsilon>0$, we construct a Lipschitz perturbation $g_\epsilon$ whose Lipschitz seminorm is less than $\epsilon$ such that every level set of $f+g_\epsilon$ has Lebesgue measure zero. 
\end{abstract}
\maketitle

\section{Introduction}

Let $I=[a,b]\subset\mathbb{R}$ and $f:[a,b]\rightarrow\mathbb{R}$. Denote by $C(I)$ the space of continuous functions on $I$ and ${\rm Lip}(I)$ the space of Lipschitz continuos functions on $I$. Recall that $f$ is called Lipschitz continuous on $I$ if  
\begin{equation}
    \nonumber
    |f|_{{\rm Lip}(I)}:=\sup_{x,y\in I, x\neq y}\frac{|f(x)-f(y)|}{|x-y|}<\infty.
\end{equation}
The objective of this note is to prove the following rather curious approximation theorem.
\begin{theorem}
    \label{LipschitzApprox}
    Let $f\in C(I)$. For any $\epsilon>0$ there exists $g_\epsilon\in{\rm Lip}(I)$ such that $|g_\epsilon|_{{\rm Lip}(I)}<\epsilon$ and for every $y\in\mathbb{R}$ there holds
    \begin{equation}
        \nonumber
        \mu(\{x\in I:f(x)+g_\epsilon(x)=y\})=0.
    \end{equation}
    ($\mu(E)$ denotes the linear Lebesgue measure of the measurable set $E$.)
\end{theorem}
In essence, Theorem \ref{LipschitzApprox} asserts that any continuous function $f$ can be perturbed by an arbitrarily small function $g_\epsilon\in{\rm Lip}(I)$ so that the perturbed function $f+g_\epsilon$ has no level set of positive measure.

Before outlining the proof of Theorem \ref{LipschitzApprox}, we briefly describe the motivation behind the result. The theorem arose naturally in the context of rearrangement inequalities \cite{BaLi26a}. The \emph{non-increasing rearrangement} of a function $f$ (see e.g. \cite[Chapter 2]{BeSh88}) is defined by
\begin{equation}
    \nonumber
    f^*(t)=\inf\left\{y>0:\mu(\{x\in I:|f(x)|>y\})<t\right\}\quad(t>0).
\end{equation}
It is well known, and easy to verify, that
$f^*$ is strictly decreasing if and only if every level set $\{x:f(x)=y\}~~(y\in\mathbb{R})$ has measure zero. An immediate consequence of Theorem \ref{LipschitzApprox} is therefore that any continuous function can be perturbed by a function with arbitrarily small Lipschitz seminorm so that the resulting function has a strictly decreasing rearrangement. In \cite{BaLi26a}, we use this observation to establish a variant of the fractional P\'{o}lya--Szeg\H{o} principle for functions on the line.

We finally outline the idea behind the proof of Theorem \ref{LipschitzApprox}. It is instructive to consider how a perturbation $g_\epsilon$ can be constructed for a concrete and reasonably non-trivial function with level sets of positive measure, such as the Cantor function $F_C$. In this case, one may define $g_\epsilon$ so that, on each interval where $F_C$ is constant, its graph forms an isosceles triangle with sufficiently small height. The height is chosen so that the slope of $g_\epsilon$ remains bounded in magnitude by $\epsilon$.
It is not difficult to convince oneself that $F_C+g_\epsilon$ has no level set with positive measure. The general case is subtler, since the level sets of an arbitrary continuous function can be quite intricate. In particular, a level set may have positive measure while still having empty interior. The construction outlined above is generalized in Lemma \ref{constructiveLemma}.  Nevertheless, the proof of Theorem \ref{LipschitzApprox} itself is non-constructive, as it relies on the Baire category theorem.

\section{Proof of Theorem \ref{LipschitzApprox}}

For $y\in\mathbb{R}$ and any function $u$, we denote
\begin{equation}
    \nonumber
    u^{-1}(y)=\{x\in I: u(x)=y\}.
\end{equation}
\begin{lem}
    \label{countLemma}
    Let $u:[a,b]\rightarrow\mathbb{R}$ and denote
    \begin{equation}
        \nonumber
        B(u)=\{y\in\mathbb{R}:\mu(u^{-1}(y))>0\}.
    \end{equation}
    Then $B(u)$ is countable.
\end{lem}
\begin{proof}
    Set $B_m(u)=\{y\in\mathbb{R}:\mu(u^{-1}(y))\ge 1/m\}$, then
\begin{equation}
    \nonumber
    B(u)=\bigcup_{m=1}^\infty B_m(u).
\end{equation}
Take $y_1,y_2,\ldots, y_N\in B_m(u)$. The sets $u^{-1}(y_1),u^{-1}(y_2),\ldots, u^{-1}(y_N)$ are pairwise disjoint and
\begin{equation}
    \nonumber
    (b-a)\ge\mu\left(\bigcup_{j=1}^Nu^{-1}(y_j)\right)=\sum_{j=1}^N\mu(u^{-1}(y_j))\ge\frac{N}{m}.
\end{equation}
Hence, $N\le m(b-a)$ so $B_m(u)$ contains at most $m(b-a)$ elements. Since each $B_m(u)$ is finite, $B(u)$ is countable.
\end{proof}
We now prove the following auxiliary result. 
\begin{lem}
    \label{constructiveLemma}
    Assume $u\in C(I)$ and let $\lambda>0$ and $\delta>0$ be arbitrary numbers. There exists a function $h$ and a measurable set $G$ with the following properties:
    \begin{enumerate}
        \item $h\in{\rm Lip}(I)$ with $|h|_{{\rm Lip}(I)}=\lambda$ and $\max_I |h(x)|\le\lambda|I|$;
        \item $\mu(G)<\delta$;
        \item if $v=u+h$, then for every $y\in\mathbb{R}$
    \begin{equation}
        \nonumber
        \mu(v^{-1}(y)\setminus G)=0.
    \end{equation}
    \end{enumerate}
\end{lem}  
\begin{proof}
    We first define the perturbation $h$ based on the level sets of $u$. Enumerate $B(u)=\{y_n\}$ and write $E_n=u^{-1}(y_n)$. Take $N\in\mathbb{N}$ such that
    \begin{equation}
        \nonumber
        \sum_{n>N}\mu(E_n)<\frac{\delta}{2}.
    \end{equation}
    Consider $E_n~~(1\le n\le N)$.
    Since $u\in C(I)$, each $E_n$ is closed, and thus each $E_n$ is compact. Consequently, for $1\le m,n\le N, n\neq m$ there holds ${\rm dist}(E_n,E_m)>0$. We may select pairwise disjoint open sets $G_n~~(1\le n\le N)$ such that $E_n\subset G_n$ and 
    \begin{equation}
        \nonumber
        \sum_{n\le N}\mu(G_n\setminus E_n)<\frac{\delta}{2}.
    \end{equation}
    For every $n$ write $G_n$ as
    \begin{equation}
        \nonumber
        G_n=\bigcup_{k=1}^\infty I_{n,k},
    \end{equation}
    where the intervals $I_{n,k}=(a_{n,k},b_{n,k})$ are pairwise disjoint. Define 
    $$
    \tau_{n,k}(x)=\max\left(0,1-\frac{2|x-m_{n,k}|}{|I_{n,k}|}\right)
    $$
    where $m_{n,k}=(a_{n,k}+b_{n,k})/2$. Set $H_{n,k}=\lambda|I_{n,k}|/2$ and define
    \begin{equation}
        \nonumber
        h_n(x)=\sum_{k=1}^\infty H_{n,k}\tau_{n,k}(x).
    \end{equation}
    For fixed $n$ the support of $\tau_{n,k}$ is $I_{n,k}$, so there is no convergence issue in the definition of $h_n$. Moreover, $h_n$ is supported on $G_n$ and $|h_n|_{{\rm Lip}(I)}=\lambda$.
    Set
    \begin{equation}
        \nonumber
        h(x)=\sum_{n\le N}h_n(x),
    \end{equation}
    then $|h|_{{\rm Lip}(I)}=\lambda$ and $h$ is supported on the union of $G_n~~(1\le n\le N)$. Since there is $x_0$ with $h(x_0)=0$, we have $|h(x)|=|h(x)-h(x_0)|\le\lambda(b-a)$. Denote
    \begin{equation}
        \nonumber
        G=\left(\bigcup_{n>N}E_n\right)\bigcup\left(\bigcup_{n\le N}(G_n\setminus E_n)\right).
    \end{equation}
    Set $v=u+h$, we shall prove that $\mu(v^{-1}(y)\setminus G)=0$ for every $y\in\mathbb{R}$. Fix $c_1\in\mathbb{R}$ and set $A:=v^{-1}(c_1)\setminus G$.
    Note that for any $x\in A$ we have $x\notin E_n$ for any $n>N$ and $x\notin G_n\setminus E_n$ for $n\le N$. Define
    \begin{equation}
        \nonumber
        A_1=\{x\in A: u(x)\notin\{y_1,y_2,\ldots y_N\}\},\quad A_2=A\setminus A_1.
    \end{equation}
    If $x\in A_1$, then by definition of $E_n$ we have that $x\notin E_n$ for $n\le N$. Therefore, $x\notin G_n$ for any $n\le N$. Consequently, $x\notin\text{supp}(h)$ and it follows that $u(x)=v(x)=c_1$ for all $x\in A_1$.
    Since $x\in A_1$, we have $c_1\neq y_n$ for all $1\le n\le N$. Further, since $x\notin G$, we also cannot have $c_1=y_n$ for some $n_0>N$, since that would imply $x\in E_{n_0}$. Hence,
    $$
    A_1\subset u^{-1}(c_1)
    $$
    and the set $u^{-1}(c_1)$ has measure 0, since $c_1\notin\{y_n\}$.
    Take now $x\in A_2$. Then $u(x)\in\{y_1,y_2,\ldots,y_N\}$ whence
    \begin{equation}
        \nonumber
        A_2\subset \bigcup_{n=1}^NE_n
    \end{equation}
    We shall prove that $\mu(A_2\cap E_n)=0$, thus showing $\mu(A_2)=0$. For every $x\in A_2\cap E_n$ there holds $v(x)=y_n+h_n(x)=c_1$. Since the equation $h_n(x)=c_1-y_n$ has at most two solutions on each component $I_{n,k}$ of $G_n$, the set $A_2\cap E_n$ is countable. Hence, $\mu(A_2\cap E_n)=0$.    
\end{proof}
\begin{proof}[Proof of Theorem \ref{LipschitzApprox}]
The proof uses Baire's category theorem (see e.g. \cite[Chapter 4]{ShSt11}). Fix $\epsilon>0$ and define 
\begin{equation}
    \nonumber
    X_\epsilon=\{g\in{\rm Lip}(I): |g|_{{\rm Lip}(I)}\le\epsilon\}.
\end{equation}
Then $X_\epsilon$ is a closed subset of the complete metric space $(C(I),d)$, where $d$ is the supremum metric
$$
d(g_1,g_2)=\max_{x\in I}|g_1(x)-g_2(x)|.
$$
Indeed, assume that $\{g_n\}\subset X_\epsilon$ and $d(g_n,g)\rightarrow0$. We must show that $g\in X_\epsilon$.
This is clear, for any $x,y\in I, x\neq y$, since
$$
\frac{|g(x)-g(y)|}{|x-y|}=\lim_{n\rightarrow\infty}\frac{|g_n(x)-g_n(y)|}{|x-y|}\le\epsilon.
$$
Define now
$$
F_m=\{g\in X_\epsilon: \text{ there is }y\in\mathbb{R}\text{ such }\mu((f+g)^{-1}(y))\ge 1/m\}.
$$
Then
\begin{enumerate}
    \item[(a)] each $F_m$ is a closed set,
    \item[(b)] no $F_m$ contains a ball.
\end{enumerate}
Fix $m$. We first prove (a). Assume that $g$ is a fixed limit point of $F_m$, we must show that $g\in F_m$. Take a sequence $\{g_n\}\subset F_m$ such that $d(g_n,g)\rightarrow0$. For each $n$, there is $y_n$ such that
\begin{equation}
    \nonumber
    \mu(\{x\in I:f(x)+g_n(x)=y_n\})\ge\frac{1}{m}.
\end{equation}
By the uniform convergence of $\{g_n\}$ to $g$, the sequence $\{y_n\}$ is contained in a bounded interval $[s_1-1,s_2+1]$ where 
$$
s_1=\min_I(f(x)+g(x))\quad\text{and}\quad s_2=\max_I(f(x)+g(x)).
$$
Therefore, we may assume (after passing to a subsequence) that $y_n$ is convergent. Let $\lim_{n\rightarrow\infty}y_n=y_0$. We shall show that
\begin{equation}
    \nonumber
    \mu(\{x\in I:f(x)+g(x)=y_0\})\ge\frac{1}{m},
\end{equation}
demonstrating that $g\in F_m$. For each $j$ define
$$
K_j=\{x\in I:y_0-1/j\le f(x)+g(x)\le y_0+1/j\}
$$
and 
$$
K_\infty=\{x\in I:f(x)+g(x)=y_0\}.
$$
Clearly 
$$
K_{j+1}\subset K_j,\quad K_\infty=\bigcap_jK_j.
$$
Fix $j$. There exists $n_j$ such that if $n\ge n_j$, then
$$
y_0-\frac{1}{2j}<y_n<y_0+\frac{1}{2j},\quad\max_{x\in I}|g(x)-g_n(x)|<\frac{1}{2j}
$$
Then it is clear that
$$
\{x\in I:f(x)+g_{n_j}(x)=y_{n_j}\}\subset K_j.
$$
Thus, $\mu(K_j)\ge1/m$. From this it follows that $\mu(K_\infty)\ge1/m$, whence $g\in F_m$. We continue to show that $F_m$ contains no open ball. Take arbitrary $g_0\in X_\epsilon$ and arbitrary $r_0>0$. We want to find $g\in X_\epsilon\cap B(g_0,r_0)$ such that $g\notin F_m$, that is to say,
\begin{equation}
    \nonumber
    \mu((f+g)^{-1}(y))<\frac{1}{m}
\end{equation}
for every $y\in\mathbb{R}$.
Applying Lemma \ref{constructiveLemma} with $u=f+g_0$ and
$$
\lambda=\min\left(\frac{r_0}{b-a}, \epsilon-|g_0|_{{\rm Lip}(I)}\right),\quad\delta=\frac{1}{m},
$$
we obtain $h\in{\rm Lip}(I)$ with $|h|_{{\rm Lip}(I)}=\lambda$, $\|h\|_{L^\infty}\le\lambda\mu(I)$, and a measurable $G$ with $\mu(G)<1/m$ such that for every $y\in\mathbb{R}$
$$
\mu((u+h)^{-1}(y)\setminus G)=0.
$$
Set now $g=g_0+h$, then $g\in X_\epsilon$. Further, $(u+h)=f+g$. Note that for every $y\in\mathbb{R}$
\begin{eqnarray}
    \nonumber
    \mu((u+h)^{-1}(y))&\le&\mu((u+h)^{-1}(y)\setminus G)+\mu((u+h)^{-1}(y)\cap G)\le 0+\mu(G)\\
    \nonumber
    &<&1/m.
\end{eqnarray}
On the other hand, $(u+h)^{-1}(y)=(f+g)^{-1}(y)$ for every $y$ whence $g\notin F_m$. Furthermore,
$$
d(g,g_0)\le\max_I|h(x)|\le\lambda|I|\le r_0,
$$
so $g\in B(g_0,r_0)$. This shows that $F_m$ has empty interior.
Since  $X_\epsilon$ is a closed subset of a complete metric space, the Baire category theorem implies that
$$
X_\epsilon\neq\bigcup_mF_m
$$
or in other words, there exists $g_\epsilon\in X_\epsilon$ such that
$$
g_\epsilon\in \left(\bigcup_mF_m\right)^C=\bigcap_m F_m^C.
$$
Since $g_\epsilon\in F_m^C$
\begin{equation}
    \nonumber
    \mu((f+g_\epsilon)^{-1}(y))<\frac{1}{m}
\end{equation}
for every $y\in\mathbb{R}$. Since this holds for every $m$,
\begin{equation}
    \nonumber
    \mu((f+g_\epsilon)^{-1}(y))=0
\end{equation}
for every $y\in\mathbb{R}$.
\end{proof}

\bibliographystyle{plain}
\bibliography{ref}
\end{document}